\documentclass[11pt]{amsart}
\usepackage{a4wide}
\usepackage{amsmath, amsfonts, amssymb,amscd,graphicx,amsthm,cite}
\usepackage{mathrsfs,url}
\usepackage{hyperref}
\usepackage{color}

\theoremstyle{plain}

    \newtheorem{thm}{Theorem}
    \newtheorem{theorem}[thm]{Theorem}

    \newtheorem{lemma}[thm]{Lemma}

    \newtheorem{question}[thm]{Question}

\theoremstyle{definition}

		\newtheorem{remark}[thm]{Remark}

\theoremstyle{remark}

\newcommand{\ignore}[1]{}

\newcommand{\Z}{{\mathbf Z}}

\newcommand{\alg}[1]{\mathbf{#1}}

\newcommand{\algA}{\alg{A}}

\DeclareMathOperator{\Eq}{\mathsf{Eq}}
\DeclareMathOperator{\Id}{\mathsf{Id}}

\DeclareMathOperator{\pEq}{\mathsf{Eq}}
\DeclareMathOperator{\pId}{\mathsf{Id}}

\DeclareMathOperator{\expo}{\mathrm exp}

\author{Michael Kompatscher}
\date{\today}
\title{Some notes on extended equation solvability and identity checking for groups}

\begin{document}

\begin{abstract}
Every finite non-nilpotent group can be extended by a term operation such that solving equations in the resulting algebra is NP-complete and checking identities is co-NP-complete. This result was firstly proven by Horv{\'a}th and Szab{\'o}; the term constructed in their proof depends on the underlying group. In this paper we provide a uniform term extension that induces hardness. In doing so we also characterize a big class of solvable, non-nilpotent groups for which extending by the commutator operation suffices.
\end{abstract}

\maketitle

\section{Introduction}

The \emph{equation solvability problem $\pEq(\algA)$} of a finite algebra $\algA$ is the computational problem of deciding whether for two polynomials $f(x_1,\ldots,x_n)$, $g(x_1,\ldots,x_n)$ over $\algA$ the equation $f(x_1,\ldots,x_n) = g(x_1,\ldots,x_n)$ has a solution in $\algA$ or not. The \emph{identity checking problem $\pId(\algA)$} (also known as the \emph{equivalence problem}) is the dual problem that asks whether for \emph{all} assignments of the variables $x_1,\ldots,x_n$ the equation $f(x_1,\ldots,x_n) = g(x_1,\ldots,x_n)$ holds.

Many results regarding the complexity of these problems are known for finite groups. If the group $G$ is nilpotent both $\pId(G)$ and $\pEq(G)$ are in P (see \cite{goldmannrussell}, \cite{BurrisLawrence-groups}, \cite{HorvathEqSolvNilp}, \cite{Attila-nilpotentEqs}). More general, so called semipattern groups induce problems that are in P \cite{Attila-semipattern}. By \cite{Horvathetal-nonsolvablegroups} the equation solvability problem of a non-solvable group is NP-complete and its identity checking problem is co-NP-complete. However the complexity of both problems still remains unclassified for general solvable, non-nilpotent groups, with the smallest group for which it is unknown being the symmetric group $S_4$ (see Problem 1 in \cite{HorvathSzabo-A4}).

Taking a different approach one might ask whether adding additional definable operations to the signature of the group has an influence on the complexity of the problem. Non-solvable groups induce hard problems, thus also extension by terms induce hard problems. By \cite{HorvathEqSolvNilp} extensions of nilpotent groups still induce problems that are in P (and the same is true in the more general setting of supernilpotent Mal'cev algebras \cite{Kompatscher-supernilpotent}). In the terminology of \cite{GorazdKrzaczkowski} the problems are \emph{representation-independent} for those algebras.

However for solvable, non-nilpotent groups the complexity may differ: It was shown in \cite{HorvathSzabo-A4} that for the alternating group $A_4$ and the commutator operation $[x,y] = x^{-1}y^{-1}xy$, $\pId(A_4;\cdot, [\cdot,\cdot])$ is co-NP-complete and $\pEq(A_4;\cdot, [\cdot,\cdot])$ is NP-complete, while both problems for $(A_4;\cdot)$ are in P. The reason for this phenomenon is that some operation can be defined in a more concise way in the extended group. For example, the length of $[\ldots [[x_1,x_2],x_3], \ldots, x_n]$ is linear in $n$ if expressed using the commutator, but $O(2^n)$ just using the group multiplication.

By \cite{HorvathSzabo-extendedgroups} every finite solvable, non-nilpotent group $(G;\cdot)$ has an extension by a term operation $f(x_1,\ldots,x_n)$ such that $\pEq(G;\cdot,f(x_1,\ldots,x_n))$ is NP-complete and $\pId(G;\cdot,f(x_1,\ldots,x_n))$ is co-NP-complete. However, the term $f$ constructed in \cite{HorvathSzabo-extendedgroups} and also its arity depend on $G$.

The first contribution of this paper is to give a uniform extension: We are going to show that the extension of a solvable, non-nilpotent groups by its commutator operation $[x,y] = x^{-1}y^{-1}xy$ and the additional 4-ary function defined by $w(x,y_1,y_2,y_3) = x^8  [[[x,y_1],y_2],y_3]$ always induces NP-hard equation solvability and co-NP-hard identity checking problems.

Secondly we prove that many cases are covered by the commutator operation alone, while for all other the extension by $w$ suffices. In doing so we proceed on answering Problem 1 of \cite{HorvathSzabo-extendedgroups}, which asks for a complete classification of the equation solvability and identity checking problem for groups extended by their commutator. Our main result states as follows:

\begin{theorem} \label{theorem:main}
Let $G$ be a finite solvable, non-nilpotent group and $L$ be the smallest group of the derived series of $G$ that is not nilpotent. By $F(L)$ we denote the Fitting subgroup of $L$. Furthermore we define the operations $[x,y] = x^{-1}y^{-1}xy$ and $w(x,y_1,y_2,y_3) = x^8  [[[x,y_1],y_2],y_3]$.
\begin{enumerate}
\item If $\expo(L/F(L)) > 2$ then $\Eq(G,\cdot,[\cdot,\cdot])$ is NP-complete and $\Id(G,\cdot,[\cdot,\cdot])$ is co-NP-complete.
\item If $\expo(L/F(L)) = 2$ then $\Eq(G,\cdot,w)$ is NP-complete and $\Id(G,\cdot,w)$ is co-NP-complete.
\end{enumerate}
\end{theorem}

We remark that (1) generalizes the hardness result for $A_4$ and the class of groups studied in Section 8.5 of \cite{GaborThesis}. A complete classification of the complexity for groups extended by their commutator is still open. But - as a consequence of our proof - it would suffice to prove that non-nilpotent dihedral groups induce hard problems in order to show hardness for \textit{all} groups in (2). We discuss this in more detail in Section \ref{sect:discussion} at the end of the paper.

In Section \ref{sect:preliminaries} we introduce notation and recall standard definitions from group theory. In Section \ref{sect:reductions} we use some complexity reduction to prove that it is enough to consider a very specific subclass of solvable, non-nilpotent groups in the proof of Theorem \ref{theorem:main}. Section \ref{sect:main} contains the proof of Theorem \ref{theorem:main}.

\section{Preliminaries} \label{sect:preliminaries}
An \textit{algebra} $\algA = (A;(f^{\algA})_{f \in \tau})$ (of \textit{type} $\tau$) consists of a set $A$ (its \textit{domain}) and a set of finitary operations $f^A: A^{ar(f)} \to A$ for every function symbol $f$ with arity $ar(f)$ in $\tau$. Often we are not going to distinguish between function symbols and their corresponding operation, but this should not cause any confusion. In this paper we only consider finite algebras, i.e. algebra of both finite domain and finite type. 

A \emph{term} $t(x_1,\ldots,x_n)$ over $\algA$ is an expression built using variables $x_1,x_2, \ldots$ and function symbols in $\tau$. The \emph{term operation} $t^\algA$ is the operation $A^n \to A$ that we obtain if we interpret every function symbol as its corresponding operation in $\algA$. \emph{Polynomials} over $\algA$ are terms, in whose construction we are additionally allowed to use constants from $A$. By the \emph{extension} of the algebra $\algA = (A;f_1^{\algA},\ldots,f_n^{\algA})$ by the term $t(x_1,\ldots,x_n)$ we denote the algebra $(\algA,t) = (A;f_1^{\algA},\ldots,f_n^{\algA},t^{\algA})$

As defined in the introduction the \emph{equation solvability problem $\pEq(\algA)$} of a finite algebra $\algA$ is defined as the computational problem of deciding whether for two polynomials $f(x_1,\ldots,x_n), g(x_1,\ldots,x_n)$ it holds that $\algA \models \exists x_1,\ldots, x_n f(x_1,\ldots,x_n) = g(x_1,\ldots,x_n)$ or not. The \emph{identity checking problem $\pEq(\algA)$} of a finite algebra $\algA$ asks whether for two polynomials $f(x_1,\ldots,x_n), g(x_1,\ldots,x_n)$ it holds that $\algA \models \forall x_1,\ldots, x_n f(x_1,\ldots,x_n) = g(x_1,\ldots,x_n)$ or not. 

Here a polynomial is encoded by a string defining it, which in a finite algebra is proportional to its \emph{length} (see e.g. \cite{AichingerMudrinskiOprsal-length} for a precise definition). We remark that in the literature one can also find other ways of encoding polynomials, which might result in different complexities (e.g. by circuits in \cite{IdziakKrzaczkowski}, or certain normal forms like in \cite{HLW-Sumofmonomialsrings}).

In this paper we are going to study finite groups $\mathbf G = (G;\cdot, e, ^{-1})$ and their term extensions. In particular $(\mathbf G, [\cdot,\cdot])$ will denote extension of the group by the \emph{commutator term} $[x,y] = x^{-1}y^{-1}xy$. For simplicities sake we are going to abuse notation and sometimes use the symbol $G$ both for the group and its underlying domain. 

The \emph{exponent $\expo(G)$} of a group is the smallest positive integer $n$ such that $g^n=e$ holds for all $g \in G$. For an element $a \in G$ we define $\langle a \rangle$ to be the smallest normal subgroup of $G$ containing $a$, i.e. the group generated by $a$ and all of its conjugated elements. For two subsets $V$, $W$ of a group $[V,W]$ will denotes the subgroup generated by all elements of the form $[v,w]$, where $v \in V$ and $w \in W$. By $G'$ we are going to denote the \emph{commutator subgroup} of $G$, i.e. $G' = [G,G]$. The \emph{centralizer} of a subset $V \subseteq G$ is defined by $C_G(V) = \{x \in G: \forall v \in V xv = vx\}$. The centralizer $C_G(G)$ is called the \emph{center} of $G$.

The \emph{derived series} of a group $G$ is defined by $G^{(1)}= G'$ and $G^{(i+1)} = [G^{(i)},G^{(i)}]$. A group is called \emph{solvable} if there is a $k$ such that $G^{(k)} = \{e\}$. The \emph{lower central series of $G$} is defined by $G_1 = G'$ and $G_{i+1} = [G,G_i]$. The \emph{upper central series of $G$} is defined by $Z_1 = C_G(G)$ and $Z_{i+1} = C_G(G/Z_i)$. A group is \emph{nilpotent} if there is a $k$ such that $G_k = \{e\}$, or equivalently $Z_k = G$.

The \emph{Fitting subgroup $F(G)$} of $G$ is the biggest nilpotent normal subgroup of $G$. It is well known \cite{BaerEngelsche} that in finite groups $F(G)$ is equal to the set of all \emph{left Engel elements}, i.e. elements $g \in G$ such that for all $h \in G$ there is an $n \in N$ with $[\cdots[[h,\underbrace{g],g],\ldots,g]}_n = e$.

\section{Some complexity reductions} \label{sect:reductions}

In this section we are going to show that in the proof of Theorem \ref{theorem:main} it suffices to consider groups $G$ that posses a nilpotent commutator subgroup, and a minimal normal subgroup $N$ of $G$ with $[G,N] = N$. The reductions we use in proving this statement slightly differ for the equation solvability problem and the identity checking problem, we are going to provide them in Lemma \ref{lemma:eqreduction} and \ref{lemma:idreduction}.

Before we start let us recall the following facts about the commutator: 

\begin{lemma} \label{lemma:facts1}
Let $G$ be a group and $N$ a normal subgroup of $G$. Then
\begin{enumerate}
\item For all $f,g,h \in G$: $[f^{-1}gf,f^{-1}hf] = f^{-1} [g,h] f$
\item For all $f \in G$, $g \in C_G(N)$, $n \in N$: $[n,f] = [n,fg] = [n,gf]$
\item If $N$ is abelian, for all $n,m \in N$, $b \in G$: $[n,b] [m,b] = [nm,b]$
\end{enumerate}
\end{lemma}

\begin{proof}
All 3 statements follow from easy computations.
\end{proof}


Let us call a subgroup $V$ of $G$ \emph{verbal} if it is the range of some term $p_V(x_1,\ldots,x_n)$ of $G$. Then the following (many-to-one) reductions hold:

\begin{lemma} \label{lemma:reductions}
Let $\mathbf V$ be a verbal subgroup of $\mathbf G$ and $f$ be a term in the language of groups. Then 
\begin{enumerate}
\item $\Eq(\mathbf V,f)$ reduces to $\Eq(\mathbf G,f)$ in polynomial time.
\item $\Id(\mathbf V,f)$ reduces to $\Id(\mathbf G,f)$ in polynomial time.
\end{enumerate}
If $V$ is a normal subgroup of $G$, let $H = G/V$, $K = G/C_G(V)$. Then 
\begin{enumerate}
\item[(3)] $\Eq(\mathbf{H},f)$ reduces to $\Eq(\mathbf{G},f)$ in polynomial time.
\item[(4)] $\Id(\mathbf {K},f)$ reduces in polynomial time to $\Id(\mathbf G,f)$.
\end{enumerate}
\end{lemma}

\begin{proof}
Lemma 9 and 10 of \cite{HorvathSzabo-extendedgroups}.
\end{proof}

Note that the commutator subgroup group of $G$ (and more general $[V,W]$ for verbal group $V$ and arbitrary $W$) is a verbal subgroup of $G$. Lemma~\ref{lemma:reductions} (1) and (2) thus imply that for every element of the derived series $L = G^{(i)}$, $\Eq(\mathbf L,f)$ reduces to $\Eq(\mathbf G,f)$ and $\Id(\mathbf L,f)$ reduces to $\Id(\mathbf G,f)$ in polynomial time. In particular this implies that in proving Theorem \ref{theorem:main} it is enough to consider the unique element of the derived series $L = G^{(i)}$ such that $L$ is not nilpotent, but $L'$ is. If we want to proceed further, we have to use the reductions Lemma~\ref{lemma:reductions} (3) and (4) respectively.

\begin{lemma} \label{lemma:eqreduction}
Let $G$ be a finite solvable, non-nilpotent group and $L$ be the last non-nilpotent element of the derived series of $G$. Then there exists an element $K$ of the lower central series of $L'$ such that $H = L / K$ has a minimal normal subgroup $N$ with $[H,N] = N$ and $[H',N]=\{e\}$. Furthermore $N$ can be chosen such that $|H/C_H(N)| > 2$ if and only if $\expo(L/ F(L)) > 2$.
\end{lemma}

\begin{proof}
Let $L$ be the last non-nilpotent element of the derived series. Since $L'$ is nilpotent and $L$ is not we have $L' \leq F(L) < L$. Let $g \in L \setminus F(L)$ and let us define the map $r_g: x \mapsto [x,g]$. Since $g \in L \setminus F(L)$, there is an element $a \in L$ such that $r_g^n(a) \neq e$ for all $n \in \mathbb N$. By the finiteness of $L$ we can without loss of generality assume that there is an $m$ such that $r_g^m(a) = a$. As $L'$ is nilpotent its lower central series is eventually equal to $\{e\}$. Let $K$ be the first element of this lower central series such that $a \notin K$. We set $H = L/K$. The normal subgroup $\langle a \rangle$ generated by $a$ in $H$ is abelian and satisfies $[H',\langle a \rangle] = \{e\}$, because it is contained in the last non-trivial element of the lower central series of $H'$. Lemma~\ref{lemma:facts1} (3) implies that the map $r_g$ is an endomorphism of $\langle a \rangle$.

An easy calculation using Lemma \ref{lemma:facts1} (1) shows $[x^{-1}ax,g] = x^{-1}[a,[x^{-1},g^{-1}]g]x$. As $H'$ lies in the centralizer of $a$, we have $r_g(x^{-1}ax) = x^{-1}r_g(a)x$. Since $a$ and all of its conjugated generate $\langle a \rangle$ and $r_g^m(a) = a$, $r_g$ has to be an automorphism of $\langle a \rangle$. Without loss of generality we can assume that $N=\langle a \rangle$ is a minimal normal subgroup - otherwise we take such minimal normal subgroup contained in $\langle a \rangle$. This concludes the proof of the first part of the Lemma.

For the second part suppose that $\expo(L / F(L)) > 2$; in other words there is a $g \in L$ such that $g, g^2 \in L \setminus F(L)$. As above we can construct the group $H$ and a minimal normal subgroup $N = \langle a \rangle$ of $H$ such that $r_{g^2}: x \mapsto [x,g^2]$ is an automorphism of $N$. We claim that also $r_{g}$ is an automorphism of $N$. Otherwise $r_{g}(m) = [m,g] = e$ would hold for some $m \neq e$. But $N$ is minimal and thus generated by $m$. This in turn implies that $N$ centralizes $g$ and consequently also $g^2$ - contradiction.

In order to prove the opposite direction suppose that $\expo(L / F(L)) = 2$, i.e. for all $g \in L$ we have that $g^2 \in F(L)$. Let $K, H$ and $N$ as above. As $K$ is nilpotent it is a subgroup of $F(L)$ and thus $L / F(L)$ is isomorphic to $H/(F(L)/K)$. Since $(F(L)/K)$ is nilpotent, it is a subgroup of $F(H)$ which implies that $\expo(H / F(H)) = 2$. By the nilpotency of $F(H)$ and minimality of $N$ we have that $[F(H),N] = \{e\}$. Therefore, for every $g \notin F(L)$ and all $a \in N$ we have $[a,g^2] = e$. An easy calculation shows $[a,g^2] = [a,g] g^{-1}[a,g]g$ which is equal to $[ag^{-1}ag,g]$ by Lemma \ref{lemma:facts1}. As $r_g$ is an automorphism of $N$, $[a,g^2] = e$ implies $g^{-1}ag = a^{-1}$. Now let $h,g \in L \setminus C(N)$. By what we showed $h^{-1}g^{-1}agh = a$ holds for all $a \in N$, hence $gh \in C(N)$. This concludes the proof that $|H / C_H(N)| = 2$.
\end{proof}

\begin{lemma} \label{lemma:idreduction}
Let $G$ be a finite solvable, non-nilpotent group and $L$ be the last non-nilpotent element of the derived series. Let $H_0 = L$ and $H_{i+1} = H_i/ C_{H_i}(H_i')$. Then there is an $i$ such that $H = H_i$ has a minimal normal subgroup $N$ with $[H,N] = N$ and $[H',N]=\{e\}$. Furthermore $N$ can be chosen such that $|H/C_H(N)| > 2$ if and only if $\expo(L / F(L)) > 2$.
\end{lemma}

\begin{proof}
Let $L$ be the last non-nilpotent element of the derived series of $G$. For $g \in L \setminus F(L)$ let us define the map $r_g: x \mapsto [x,g]$. As $g \in L \setminus F(L)$ there is an element $a \in L$ such that $r_g^n(a) \neq e$ for all $n \in \mathbb N$. By the finiteness of $L$ we can furthermore assume that there is an integer $m$ with $r_g^m(a) = a$. Let $H_0 = L$ and $H_{i+1} = H_i/ C_{H_i}(H_i')$. As $L'$ is nilpotent and $a \in L'$ there is an index $i$ such that $aH_{i-1} \in C_{H_i}(H_i') \setminus \{e\}$. We set $H = H_i$. The normal subgroup $\langle a \rangle$ generated by $a$ in $H$ is abelian, as it is contained in the center of $H'$. Lemma~\ref{lemma:facts1} (3) implies that the map $r_g$ is an endomorphism of $\langle a \rangle$. As in the proof of Lemma \ref{lemma:eqreduction} we can show that $r_g$ is actually an automorphism of $\langle a \rangle$. Furthermore we can assume that $N=\langle a \rangle$ is a minimal normal subgroup of $H$ - otherwise we take such minimal normal subgroup contained in $\langle a \rangle$. This concludes the proof of the first part of the Lemma.

The proof of the second part is also analogous to the proof of Lemma \ref{lemma:eqreduction} (using the fact that $C_{H_i}(H_i')$ is always nilpotent, as $[[C_{H_i}(H_i'),C_{H_i}(H_i')],C_{H_i}(H_i')] \leq [H_i',C_{H_i}(H_i')] = \{e\}$).
\end{proof}

\begin{remark}
We remark that the group $H$ constructed in Lemma \ref{lemma:eqreduction} and \ref{lemma:idreduction} might differ for the same $G$, even using the same elements $g,a \in G$. Further note that forming the quotients in Lemma \ref{lemma:eqreduction} and \ref{lemma:idreduction} is indeed a necessary step in proving the existence of a minimal normal subgroup $N$ with $[H,N] = N$: An example of a group without this property is the special linear group $SL(2,3)$. It is not nilpotent and its commutator subgroup is the (nilpotent) quaternion group, but its only minimal normal subgroup $\mathbb Z_2$ is equal to its center.
\end{remark}

Both Lemma \ref{lemma:eqreduction} and \ref{lemma:idreduction} reduce the task of proving Theorem \ref{theorem:main} to a very specific subclass of solvable, not nilpotent groups. In the following lemma we list some of the properties of such groups that were partially already proven in Lemma \ref{lemma:eqreduction}.

\begin{lemma} \label{lemma:facts2} Let $G$ be a finite solvable group and $N$ a minimal normal subgroup, such that $[G,N] = N$ and $[G',N] = \{e\}$. Then
\begin{enumerate}
\item $N$ is elementary abelian.
\item For all $b \in G \setminus C_G(N)$ the map $r_b: n \mapsto [n,b]$ is an automorphism of $N$.
\item If $|G / C_G(N)| = 2$, then for all $b \in G \setminus C_G(N)$ and $n \in N$ we have $b^{-1}nb = n^{-1}$. \\ This implies further that $N$ is a cyclic group of odd prime order.
\end{enumerate}
\end{lemma}

\begin{proof}
(1) is a well-known fact that holds for all minimal normal subgroups in all finite solvable groups. (2) is analogous to the proof step of Lemma \ref{lemma:eqreduction}, in which we show that $r_g$ is an automorphism of $N$. To show that $b^{-1}nb = n^{-1}$ in (3) we can also argue as in the proof of Lemma \ref{lemma:eqreduction}. If $\exp(N) = 2$ this would imply $b^{-1}nb = n$, thus $b$ commutes with $N$ - contradiction! So by (1) $\exp(N)$ has to be equal to some odd prime $p$. Since for all $b \in G \setminus C_G(N)$ $[n,b] = n^{-2}$, the commutator group $N = [G,\{n\}]$ just consists of powers of $n$, i.e it is cyclic.
%
\end{proof}

\section{Proof of the main theorem} \label{sect:main}

We are now ready to prove Theorem \ref{theorem:main}. Depending on the index of $C_G(N)$ in $G$ we do a case distinction.

\begin{lemma} \label{lemma:odd}
Let $\mathbf G = (G;\cdot,e,^{-1})$ be a finite solvable and non-nilpotent group. Furthermore assume that there is a non-trivial minimal normal subgroup $N$ such that $C_G(N) \geq G'$ and $|G/C_G(N)| > 2$. Then $\Eq(\mathbf G, [\cdot,\cdot])$ is NP-complete and $\Id(\mathbf G, [\cdot,\cdot])$ is co-NP-complete.
\end{lemma}

\begin{proof} We prove the statement by constructing terms in $(\mathbf G,{[\cdot,\cdot]})$ that allow us to reduce the $|G/C_G(N)|$-graph-coloring problem to $\Eq(\mathbf G,{[\cdot,\cdot]})$ and its complement to $\Id(\mathbf G,{[\cdot,\cdot]})$. Reductions of similar type were already used in \cite{HorvathSzabo-A4} and \cite{IdziakKrzaczkowski}.

%
Let us define the sequence of terms
\begin{align*}
t_1(x,y_1) &= [x,y_1]\\
t_{k+1}(x,y_1,\ldots,y_n,y_{k+1}) &= [t_n(x,y_1,\ldots,y_k),y_{k+1}].
\end{align*}
Note that the length of $t_k$ in $(\mathbf G,{[\cdot,\cdot]})$ grows linearly in $k$. Furthermore, if for every $j = 1,\ldots,k$ we have $b_j \in G \setminus C_G(N)$, then $x \mapsto t_k(x,b_1,\ldots,b_k)$ is an automorphism of $N$ by Lemma~\ref{lemma:facts2} (2). If there is a $j$ with $b_j \in C_G(N)$, then $t_k(x,b_1,\ldots,b_k) = e$ for all $x \in N$.

%



Now let $(V;E)$ be an undirected graph, i.e. an input of the $|G/C_G(N)|$-coloring problem. Then we define the term 
$$r(x,(y_v)_{v\in V}) =  t_{|E|}\left(x,(y_{v_1} \cdot y_{v_2}^{-1})_{(v_1,v_2)\in E}\right).$$

By the properties of $t_{|E|}$, $r(x,(b_v)_{v\in V})$, for $x \in N$ is equal to $e$ if there is an edge $(v_1,v_2) \in E$, such that $b_{v_1} b_{v_2}^{-1} \in C_G(N)$. If however, for all edges $(v_1,v_2) \in E$ the elements $b_{v_1}$ and $b_{v_2}$ are in different cosets of $C_G(N)$, then $x \mapsto r(x,(b_v)_{v\in V})$ is a permutation of $N$. Note also, that for $x \in N$ the value of $r(x,(b_v)_{v\in V})$ only depends on the $C_G(N)$-cosets of the elements $b_v$.

Thus if the equation $r(x,(y_v)_{v\in V}) = n$ for some $n \in N \setminus \{e\}$ has a solution with $x \in N$, then the graph is $|G/C_G(N)|$-colorable: just color the vertices by the corresponding $C_G(N)$ cosets of the solution to the equation. Conversely a proper $|G/C_G(N)|$-coloring of the graph gives rise to a solution of the equation, by assigning representatives of the $G/C_G(N)$-classes to the $|G/C_G(N)|$-many colors. Analogously the identity $r(x,(y_v)_{v\in V}) = e$ holds for all $x \in N$ if there is no proper $|G/C_G(N)|$-coloring of the graph $(V;E)$. 

At last notice that we can restrict ourselves to solutions with $x \in N$, since $N$ is verbal. By $[N,G]=N$ every element of $N$ can be written as the product of commutator expressions $[n,g]$, $n \in N$. Thus there is a polynomial of the form $s_N(z_1, \ldots, z_k) = [n_1,z_1] \cdot [n_2,z_2] \cdots  [n_k,z_k]$ with $n_1, \ldots, n_k \in N$ such that $N$ is its range. When substituting $x = s_N(\bar z)$ in $r$, the resulting term is still polynomial in the size of the graph $(V;E)$. Thus the $|G/C_G(N)|$-coloring problem reduces to $\Eq(\mathbf G)$ and its complement reduces to $\Id(\mathbf G)$. This concludes the proof.
\end{proof}

It is left to consider the case where $|G/C_G(N)| = 2$; we show that then $(\mathbf G,w)$ induces hard problems.

\begin{lemma} \label{lemma:even}
Let $G$ be a finite solvable, non-nilpotent group and let $w(x,y_1,y_2,y_3) = x^8([[x,y_1],y_2],y_3])$. Furthermore assume that there is a non-trivial minimal normal subgroup $N$ such that $C_G(N) \geq G'$ and $|G/C_G(N)| = 2$. Then $\Eq(\mathbf G,w)$ is NP-complete and $\Id(\mathbf G,w)$ is co-NP-complete.
\end{lemma}

\begin{proof}
%

By Lemma \ref{lemma:facts2} (3) we have for $b \notin C_G(N)$, $x \in N$ we have $b^{-1}xb = x^{-1}$ and consequently $[x,b] = x^{-2}$. We are going to prove the NP-hardness of $\Eq(\mathbf G,w)$ by encoding 3-SAT. For that observe that, by the above
\begin{align*}
w(e,y_1,y_2,y_3) &= e\\
w(x,y_1,y_2,y_3) &= e \text{ if } x \in N \text{ and } \forall~i: y_i \notin C_G(N)\\
w(x,y_1,y_2,y_3) &= x^8 \text{ if } x \in N \text{ and } \exists~i: y_i \in C_G(N)
\end{align*}

Thus if one of the elements $b_1,b_2,b_3$ lies $C_G(N)$, the map $x \mapsto w(x,b_1,b_2,b_3)$ is an automorphism of $N$. Otherwise it is constant and equal to $e$. Now let us define recursively the terms 
\begin{align*}
w_1(x,y_1,y_2,y_3) &= w(x,y_1,y_2,y_3),\\
w_{n+1}(x,y_1,y_2,\ldots y_{3n+3}) &= w(w_{n}(x,y_1,y_2,\ldots y_{3n}), y_{3n+1}, y_{3n+2}, y_{3n+3}).
\end{align*}
Then $x \mapsto w_n(x,b_1,\ldots b_{3n})$ is an automorphism of $N$, if in every triple $(b_{3i},b_{3i+1},b_{3i+2})$ at least one entry is from $C_G(N)$, otherwise it is constant and equal to $e$ (on $N$).

This allows us to encode an input of 3-SAT as an input of $\Eq(\mathbf G,w)$. Let $(l_1 \lor l_2 \lor l_3) \land \cdots \land (l_{3n} \lor l_{3n+1} \lor l_{3n+2})$ be a 3-CNF formula in variables $z_1, \ldots, z_k$ with the $l_i$ being the literals in variables $z_{i_j}$. Then let us form the polynomial $w_{n+1}(x, l_1', l_2',\ldots, l'_{3n+3})$ in variables $z_1', \ldots, z_k'$  with $ l_i' =  z_{i_j}'$ if $l_i = z_{i_j}$ and $ l_i' = b z_{i_j}'$ if $l_i = \neg z_{i_j}$, where $b$ is an element of $G \setminus C_G(N)$.

Let $n \in N \setminus \{e\}$. Then it is not hard to see that the equation $n = w_{n+1}(x, l_1', l_2',\ldots, l'_{3n+3})$ has a solution with $x \in N$ if and only if the instance of 3-SAT is satisfiable. Moreover $w_{n+1}(x, l_1', l_2',\ldots, l'_{3n+3})$ is constantly equal to $e$ for $x \in N$, if and only if the instance of 3-SAT is unsatisfiable.

$N$ is verbal, so as in the proof of Lemma \ref{lemma:odd} we can restrict us to $x \in N$, by substituting $x$ by a polynomial with range $N$. Thus 3-SAT reduces to $\Eq(\mathbf G,w)$ and its complement reduces to $\Id(\mathbf G,w)$ in polynomial time. We conclude that $\Eq(\mathbf G,w)$ is NP-complete, and $\Id(\mathbf G,w)$ is co-NP-complete.
\end{proof}

We remark that a proof of Lemma \ref{lemma:even} for the special case $G = S_3$ due to Idziak already appeared in \cite{GorazdKrzaczkowski}.
We are now ready to sum up the proof of Theorem~\ref{theorem:main}.

\begin{proof}[Proof of Theorem~\ref{theorem:main}]
Let $G$ be a finite solvable, non-nilpotent group and let $L$ be the last non-nilpotent element of its derived series. By Lemma \ref{lemma:reductions} (1),(3) and Lemma \ref{lemma:eqreduction} there is a group $H$ such that 
\begin{itemize}
\item For every term $f$, $\Eq(\mathbf H,f)$ reduces to $\Eq(\mathbf G,f)$ in polynomial time.
\item $H$ has a non-trivial minimal normal subgroup $N$ with $[H,N] = N$ and $[H',N]=\{e\}$,
\item $N$ can be picked such that $|H/C_H(N)| > 2$ if and only if $\expo(L/F(L)) > 2$.
\end{itemize}

If $|H/C_H(N)| > 2$ it follows from Lemma \ref{lemma:odd} that $\Eq(\mathbf H,[\cdot,\cdot])$ is NP-hard. If $|H/C_H(N)| = 2$, it follows from Lemma \ref{lemma:even} that $\Eq(\mathbf H,w)$ is NP-hard. The proof for the identity checking problem is analogous, using the reduction from Lemma \ref{lemma:reductions} (2),(4) and Lemma \ref{lemma:idreduction} instead.

Note that the inversion $^{-1}$ is actually not needed, since in any input we can substitute $(xy)^{-1} = y^{-1}x^{-1}$ and $x^{-1} = x^{|G|-1}$.
\end{proof}

\section{Discussion} \label{sect:discussion}

Theorem \ref{theorem:main} does not complete the complexity classification of $\Eq(\mathbf {G},[\cdot,\cdot])$ and $\Id(\mathbf {G},[\cdot,\cdot])$ for all finite groups $G$, as asked for in Problem 1 of \cite{HorvathSzabo-extendedgroups}. However the class $\mathcal C$ of groups for which the complexity is still unknown consists only of groups $G$ satisfying $G' \leq F(G) < G$ and $\exp(G/F(G)) = 2$ (and groups containing such $G$ in their derived series).

By Lemma~\ref{lemma:eqreduction} and Lemma~\ref{lemma:idreduction} it would suffices to show hardness only for all those $G \in \mathcal C$ that additionally have a minimal normal subgroup $N$ with $|G/C_G(N)| = 2$ in order to prove hardness for all $G \in \mathcal C$. By Lemma \ref{lemma:facts2} (3) the elements of $G \setminus C_G(N)$ then act on $N$ as by inversion; furthermore $N$ is a cyclic group of order $p$ for some odd prime $p$. Thus the semidirect product of $G/C_G(N)$ and $N$ is equal to the dihedral group $D_{2p}$. This naturally leads to the question:

\begin{question} \label{question:dihedral}
Let $p$ be an odd prime. What is the computational complexity of $\Eq(D_{2p},\cdot, [\cdot,\cdot])$ and $\Id(D_{2p}, \cdot, [\cdot,\cdot])$ for the dihedral group $D_{2p}$? In particular, what is the complexity for the symmetric group on 3 elements $D_6 = S_3$? (cf. Problem 1 in \cite{HorvathSzabo-extendedgroups}).
\end{question}

If $\Eq(D_{2p},[\cdot,\cdot])$ was NP-complete (respectively $\Id(D_{2p},[\cdot,\cdot])$ co-NP-complete) the proof would most likely rely on the action of $D_{2p} / \mathbb Z_p$ on $\Z_p$ and thus (using similar arguments as in Lemma \ref{lemma:even}) lift to all non-nilpotent groups with $\exp(G/F(G)) = 2$. But also other possible answers to Question \ref{question:dihedral} would be of central importance in getting a better understanding the expressive power of the commutator operation in finite groups.

\bibliographystyle{alpha}
\bibliography{kompatscher_equation_solvability}
\end{document}